\documentclass{amsart}
\usepackage[margin=1in]{geometry}
\usepackage[all]{xy}
\usepackage[normalem]{ulem}
\usepackage{verbatim}
\usepackage{color}
\usepackage{amsthm}
\usepackage{amssymb}
\usepackage[colorlinks=true]{hyperref}
\usepackage[multiple]{footmisc}

\numberwithin{equation}{section}

\newtheorem{theorem}[equation]{Theorem}
\newtheorem*{theorem*}{Theorem} \newtheorem{lemma}[equation]{Lemma}

\newtheorem*{conjecture*}{Mamma Conjecture}
\newtheorem*{conjecture1*}{Mamma Conjecture (revisited)}
\newtheorem{proposition}[equation]{Proposition}
\newtheorem{corollary}[equation]{Corollary}
\newtheorem*{corollary*}{Corollary}

\theoremstyle{remark}
\newtheorem{definition}[equation]{Definition}

\newtheorem{example}[equation]{Example}

\newtheorem{notation}[equation]{Notation}

\theoremstyle{remark}
\newtheorem{remark}[equation]{Remark}

\newcommand{\cA}{{\mathcal A}}
\newcommand{\cB}{{\mathcal B}}
\newcommand{\cC}{{\mathcal C}}
\newcommand{\cD}{{\mathcal D}}
\newcommand{\cE}{{\mathcal E}}
\newcommand{\cF}{{\mathcal F}}

\newcommand{\cL}{{\mathcal L}}

\newcommand{\cO}{{\mathcal O}}
\newcommand{\cP}{{\mathcal P}}

\newcommand{\cT}{{\mathcal T}}

\newcommand{\bbA}{\mathbb{A}}
\newcommand{\bbB}{\mathbb{B}}
\newcommand{\bbC}{\mathbb{C}}

\newcommand{\bbF}{\mathbb{F}}

\newcommand{\bbP}{\mathbb{P}}

\newcommand{\bbQ}{\mathbb{Q}}
\newcommand{\bbZ}{\mathbb{Z}}

\DeclareMathOperator{\id}{id}

\newcommand{\dgcat}{\mathrm{dgcat}} 

\newcommand{\perf}{\mathrm{perf}}

\newcommand{\dg}{\mathrm{dg}}

\newcommand{\Hom}{\mathrm{Hom}}

\newcommand{\rep}{\mathrm{rep}}

\newcommand{\Hmo}{\mathrm{Hmo}}
\newcommand{\op}{\mathrm{op}}

\newcommand{\too}{\longrightarrow}

\let\oldmarginpar\marginpar
\def\marginpar#1{\oldmarginpar{\tiny #1}}

\begin{document}

\title[]{Noncommutative Riemann hypothesis}
\author{Gon{\c c}alo~Tabuada}
\address{Gon{\c c}alo Tabuada, Mathematics Institute, Zeeman Building, University of Warwick, Coventry CV4 7AL UK.}
\email{goncalo.tabuada@warwick.ac.uk}
\urladdr{https://homepages.warwick.ac.uk/staff/Goncalo.Tabuada/}

\date{\today}

\abstract{In this note, making use of noncommutative $l$-adic cohomology, we extend the generalized Riemann hypothesis from the realm of algebraic geometry to the broad setting of geometric noncommutative schemes in the sense of Orlov. As a first application, we prove that the generalized Riemann hypothesis is invariant under derived equivalences and homological projective duality. As a second application, we prove the noncommutative generalized Riemann hypothesis in some new cases.}}

\maketitle



\section{Introduction and statement of results}\label{sec:intro}
Let $k$ be a global field and $\Sigma_k$ its (infinite) set of non-archimedean places.

\smallskip

Let $X$ be a smooth proper $k$-scheme and $0\leq w \leq 2\mathrm{dim}(X)$ an integer. Following Serre's foundational work \cite{Serre1,Serre2} (consult also Manin \cite{Manin}), consider the $L$-function $L_w(X;s):=\prod_{\nu \in \Sigma_k} L_{w, \nu}(X;s)$ of weight $w$. As proved in {\em loc. cit.}, this infinite product converges absolutely in the half-plane $\mathrm{Re}(s)> \frac{w}{2} +1$ and is non-zero in this region. Moreover, the following two conditions are expected to hold:

\smallskip

\noindent
$(\mathrm{C}1)$ The $L$-function $L_w(X;s)$ admits a (unique) meromorphic continuation to the entire complex plane.
\noindent
$(\mathrm{C}2)$ When $\mathrm{char}(k)=0$, the only possible pole of $L_w(X;s)$ is located at $s=\frac{w}{2}+1$ with $w$ even.

\smallskip

When $\mathrm{char}(k)=0$, the conditions $(\mathrm{C}1)\text{-}(\mathrm{C}2)$ have been proved in many cases: certain $0$-dimensional schemes, certain elliptic curves, certain modular curves, certain abelian varieties, certain varieties of Fermat type, certain Shimura varieties, etc. When $\mathrm{char}(k)>0$, condition $(\mathrm{C}1)$ follows from Grothendieck's work~\cite{Grothendieck}. 

The following conjecture, which implicitly assumes condition $(\mathrm{C}1)$, goes back to the work \cite{Riemann} of Riemann. 

\smallskip

\noindent
\underline{{\it Generalized Riemann hypothesis}} $\mathrm{R}_w(X)$: {\it All the zeros of the $L$-function $L_w(X;s)$ that are contained in the critical strip $\frac{w}{2}< \mathrm{Re}(s) < \frac{w}{2}+1$ lie in the vertical line $\mathrm{Re}(s)=\frac{w+1}{2}$.}

\smallskip

The generalized Riemann hypothesis play a central role in mathematics. For example, when $\mathrm{char}(k)=0$ and $X=\mathrm{Spec}(k)$, the conjecture $\mathrm{R}_0(X)$ reduces to the classical {\em extended Riemann hypothesis} $\mathrm{ERH}_k$, i.e., all the zeros of the Dedekind zeta function $\zeta_k(s):=\sum_{I \triangleleft \cO_k}\frac{1}{N(I)^s}$ that are contained in the critical strip $0 < \mathrm{Re}(s) < 1$ lie in the vertical line $\mathrm{Re}(s)=\frac{1}{2}$; note that in the particular case where $k=\bbQ$, $\mathrm{ERH}_k$ is the famous Riemann hypothesis. The status of the generalized Riemann conjecture depends drastically on the characteristic of $k$. On the one hand, when $\mathrm{char}(k)=0$, no cases have been proved. On the other hand, when $\mathrm{char}(k)>0$, the generalized Riemann hypothesis follows from Deligne's work \cite{Deligne,Weil1}.  

\smallskip

Now, let $\cA$ be a geometric noncommutative $k$-scheme in the sense of Orlov; consult \S\ref{sub:Orlov} below. A standard example is the canonical dg enhancement $\perf_\dg(X)$ of the derived category of perfect complexes $\perf(X)$ of a smooth proper $k$-scheme $X$ (consult Keller's survey \cite[\S4.6]{Keller}); consult \S\ref{sec:applications-2} below for further examples. In \S\ref{sec:NC-L} below, making use of noncommutative $l$-adic cohomology, we construct the noncommutative counterparts $L_{\mathrm{even}}(\cA;s):=\prod_{\nu \in \Sigma_k} L_{\mathrm{even}, \nu}(\cA;s)$ and $L_{\mathrm{odd}}(\cA;s):=\prod_{\nu \in \Sigma_k} L_{\mathrm{odd}, \nu}(\cA;s)$ of the classical $L$-functions. Moreover, we prove the following noncommutative counterpart of Serre's convergence result:
\begin{theorem}\label{thm:convergence}
The infinite product $L_{\mathrm{even}}(\cA;s)$, resp. $L_{\mathrm{odd}}(\cA;s)$, converges absolutely in the half-plane $\mathrm{Re}(s)>1$, resp. $\mathrm{Re}(s)>\frac{3}{2}$, and is non-zero in this region.
\end{theorem}
Similarly to the above condition $\mathrm{(C1)}$, it is expected that the noncommutative $L$-functions $L_{\mathrm{even}}(\cA;s)$ and $L_{\mathrm{odd}}(\cA;s)$ admit a (unique) meromorphic continuation to the entire complex plane. Under this assumption, the generalized Riemann hypothesis admits the following noncommutative counterpart:

\medskip

\noindent
\underline{{\it Noncommutative generalized Riemann hypothesis}} $\mathrm{R}_{\mathrm{even}}(\cA)$ and $\mathrm{R}_{\mathrm{odd}}(\cA)$: {\it All the zeros of the noncommutative $L$-function $L_{\mathrm{even}}(\cA;s)$, resp. $L_{\mathrm{odd}}(\cA;s)$, that are contained in the critical strip $0< \mathrm{Re}(s) < 1$, resp. $\frac{1}{2}< \mathrm{Re}(s) < \frac{3}{2}$, lie in the vertical line $\mathrm{Re}(s)=\frac{1}{2}$, resp. $\mathrm{Re}(s)=1$.}

\smallskip

The noncommutative (generalized) Riemann hypothesis was originally envisioned by Kontsevich in his seminal talks \cite{Hodge,IAS}. The next result relates this conjecture with the generalized Riemann hypothesis:

\begin{theorem}\label{thm:main}
Given a smooth proper $k$-scheme $X$, we have the following implications:
\begin{eqnarray}\label{eq:implications}
\{\mathrm{R}_w(X)\}_{w\, \mathrm{even}} \Rightarrow \mathrm{R}_{\mathrm{even}}(\perf_\dg(X)) && \{\mathrm{R}_w(X)\}_{w\, \mathrm{odd}} \Rightarrow \mathrm{R}_{\mathrm{odd}}(\perf_\dg(X))\,.
\end{eqnarray}
When $\mathrm{char}(k)>0$, the converse implications of \eqref{eq:implications} hold. Moreover, when $\mathrm{char}(k)=0$ and the $L$-functions $\{L_w(X;s)\}_{0\leq w \leq 2\mathrm{dim}(X)}$ satisfy condition $(\mathrm{C}2)$, the converse implications of \eqref{eq:implications} also hold. 
\end{theorem}
Intuitively speaking, Theorem \ref{thm:main} shows that the generalized Riemann hypothesis belongs not only to the realm of algebraic geometry but also to the broad setting of geometric noncommutative schemes.

\section{Applications to commutative geometry}\label{sec:applications-1}
Let $k$ be a global field; since the generalized Riemann hypothesis holds when $\mathrm{char}(k)>0$, we restrict ourselves to the case where $\mathrm{char}(k)=0$. 
In this section, making use of Theorem \ref{thm:main}, we prove that the generalized Riemann hypothesis is invariant under derived equivalences and homological projective duality.
\subsection*{Derived invariance}
Let $X$ and $Y$ be two smooth proper $k$-schemes. In what follows, we assume that the associated $L$-functions $\{L_w(X;s)\}_{0 \leq w \leq 2\mathrm{dim}(X)}$ and $\{L_w(Y;s)\}_{0 \leq w \leq 2\mathrm{dim}(Y)}$ satisfy condition $\mathrm{(C2)}$. 
\begin{corollary}[Derived invariance]\label{cor:invariance}
If the derived categories of perfect complexes $\mathrm{perf}(X)$ and $\mathrm{perf}(Y)$ are (Fourier-Mukai) equivalent, then we have the following equivalences:
\begin{eqnarray}\label{eq:searched}
\{\mathrm{R}_w(X)\}_{w\,\mathrm{even}} \Leftrightarrow \{\mathrm{R}_w(Y)\}_{w\,\mathrm{even}} && \{\mathrm{R}_w(X)\}_{w\,\mathrm{odd}} \Leftrightarrow \{\mathrm{R}_w(Y)\}_{w\,\mathrm{odd}}\,.
\end{eqnarray} 
\end{corollary}
\begin{proof}
If the triangulated categories $\perf(X)$ and $\perf(Y)$ are (Fourier-Mukai) equivalent, then the dg categories $\perf_\dg(X)$ and $\perf_\dg(Y)$ are Morita equivalent. Consequently, we obtain the following equivalences: 
\begin{eqnarray*}
\mathrm{R}_{\mathrm{even}}(\perf_\dg(X)) \Leftrightarrow \mathrm{R}_{\mathrm{even}}(\perf_\dg(Y)) && \mathrm{R}_{\mathrm{odd}}(\perf_\dg(X)) \Leftrightarrow \mathrm{R}_{\mathrm{odd}}(\perf_\dg(Y))\,.
\end{eqnarray*}
By combining them with Theorem \ref{thm:main}, we hence obtain the above equivalences \eqref{eq:searched}.
\end{proof}
In the literature there are numerous examples of smooth proper $k$-schemes $X$ and $Y$ for which the above Corollary \ref{cor:invariance} applies; consult, for example, the book \cite{Huybrecht} and the references therein.
\subsection*{Homological Projective Duality}
Homological Projective Duality (=HPD) was introduced by Kuznetsov in \cite{KuznetsovHPD} as a tool to study the derived categories of perfect complexes of linear sections. Let $X$ be a smooth proper $k$-scheme equipped with a line bundle $\cL_X(1)$; we write $X\to \bbP(V)$ for the associated map, where $V:=H^0(X,\cL_X(1))^\vee$. Assume that we have a Lefschetz decomposition $\perf(X)=\langle \bbA_0, \bbA_1(1), \ldots, \bbA_{i-1}(i-1)\rangle$ with respect to $\cL_X(1)$ in the sense of \cite[Def.~4.1]{KuznetsovHPD}. Following \cite[Def.~6.1]{KuznetsovHPD}, let us write $Y$ for the HP-dual of $X$, $\cL_Y(1)$ for the HP-dual line bundle, and $Y \to \bbP(V^\vee)$ for the associated map. Given a generic linear subspace $L \subset V^\vee$, consider the smooth linear sections $X_L:=X\times_{\bbP(V)}\bbP(L^\perp)$ and $Y_L:=Y \times_{\bbP(V^\vee)}\bbP(L)$. In what follows, we assume that the associated $L$-functions $\{L_w(X_L;s)\}_{0 \leq w \leq 2\mathrm{dim}(X_L)}$ and $\{L_w(Y_L;s)\}_{0 \leq w \leq 2\mathrm{dim}(Y_L)}$ satisfy the above condition $\mathrm{(C2)}$.
\begin{theorem}[HPD-invariance]\label{thm:HPD}
Assume that the triangulated category $\bbA_0$ admits a full exceptional collection. Under this assumption, the following holds:
\begin{eqnarray}\label{eq:equivalences-main}
\mathrm{ERH}_k \Rightarrow \big( \{\mathrm{R}_w(X_L)\}_{w\, \mathrm{even}} \Leftrightarrow  \{\mathrm{R}_w(Y_L)\}_{w\, \mathrm{even}} \big)  && \{\mathrm{R}_w(X_L)\}_{w\, \mathrm{odd}} \Leftrightarrow  \{\mathrm{R}_w(Y_L)\}_{w\, \mathrm{odd}}\,.
\end{eqnarray} 
\end{theorem}
\begin{remark}
The assumption of Theorem \ref{thm:HPD} is quite mild since it holds in all the examples in the literature.
\end{remark}
Roughly speaking, Theorem \ref{thm:HPD} ``cuts in half'' the difficulty of proving the generalized Riemann hypothesis, i.e., if the generalized Riemann hypothesis holds for a linear section, then it also holds for the HP-dual linear section. In the literature there are numerous examples of homological projective dualities for which the above Theorem \ref{thm:HPD} applies (e.g., Veronese-Clifford duality, Grassmannian-Pfaffian duality, Spinor duality, Determinantal duality, etc); consult, for example, the surveys \cite{KuznetsovICM,Thomas} and the references therein.
\section{Applications to noncommutative geometry}\label{sec:applications-2}
Let $k$ be a global field. In this section, making use of Theorem \ref{thm:main}, we prove the noncommutative generalized Riemann hypothesis in some new cases.

\subsection*{Noncommutative gluings of schemes} Let $X$ and $Y$ be two smooth proper $k$-schemes and $\mathrm{B}$ a perfect dg $\perf_\dg(X)\text{-}\perf_\dg(Y)$ bimodule. Following Orlov \cite[Def.~3.5]{Orlov1}, we can consider the gluing $X \odot_{\mathrm{B}} Y$ of the dg categories $\perf_\dg(X)$ and $\perf_\dg(Y)$ via the dg bimodule $\mathrm{B}$ (Orlov used a different notation). As proved by Orlov in \cite[Thm.~4.11]{Orlov1}, $X\odot Y$ is a geometric noncommutative $k$-scheme.
\begin{theorem}\label{thm:gluing}
We have the following implications:
\begin{eqnarray*}
\{\mathrm{R}_w(X)\}_{w\, \mathrm{even}} + \{\mathrm{R}_w(Y)\}_{w\, \mathrm{even}} \Rightarrow \mathrm{R}_{\mathrm{even}}(X\odot_{\mathrm{B}} Y) && \{\mathrm{R}_w(X)\}_{w\, \mathrm{odd}} + \{\mathrm{R}_w(Y)\}_{w\, \mathrm{odd}} \Rightarrow \mathrm{R}_{\mathrm{odd}}(X\odot_{\mathrm{B}} Y)\,.
\end{eqnarray*}
In particular, the conjectures $ \mathrm{R}_{\mathrm{even}}(X\odot_{\mathrm{B}} Y)$ and $ \mathrm{R}_{\mathrm{odd}}(X\odot_{\mathrm{B}} Y)$ hold when $\mathrm{char}(k)>0$.
\end{theorem}
\subsection*{Calabi-Yau dg categories associated to hypersurfaces} Let $X \subset \bbP^n$ be a smooth hypersurface of degree $\mathrm{deg}(X) \leq n+1$. As proved by Kuznetsov in \cite[Cor.~4.1]{KuznetsovCY}, we have a semi-orthogonal decomposition $\perf(X) = \langle \cT, \cO_X, \ldots, \cO_X(n- \mathrm{deg}(X))\rangle$. Moreover, the full dg subcategory $\cT_\dg$ of $\perf_\dg(X)$, consisting of the objects of $\cT$, is a Calabi-Yau dg category of fractional CY-dimension $\frac{(n+1)(\mathrm{deg}(X)-2)}{\mathrm{deg}(X)}$. Note that $\cT_\dg$ is a geometric noncommutative $k$-scheme. Note also that $\cT_\dg$ is {\em not} Morita equivalent to a dg category of the form $\mathrm{perf}_\dg(Y)$, with $Y$ a smooth proper $k$-scheme, whenever its CY-dimension is {\em not} an integer.
\begin{theorem}\label{thm:CY}
We have the following implications:
\begin{eqnarray}\label{eq:implications1}
\{\mathrm{R}_w(X)\}_{w\, \mathrm{even}} \Rightarrow \mathrm{R}_{\mathrm{even}}(\cT_\dg) && \{\mathrm{R}_w(X)\}_{w\, \mathrm{odd}} \Rightarrow \mathrm{R}_{\mathrm{odd}}(\cT_\dg)\,.
\end{eqnarray}
In particular, the conjectures $\mathrm{R}_{\mathrm{even}}(\cT_\dg)$ and $ \mathrm{R}_{\mathrm{odd}}(\cT_\dg)$ hold when $\mathrm{char}(k)>0$.
\end{theorem}

\subsection*{Finite-dimensional algebras of finite global dimension} Let $A$ be a finite-dimensional $k$-algebra of finite global dimension. Examples include path algebras of finite quivers without oriented cycles and their admissible quotients. As proved by Orlov in \cite[Cor.~5.4]{Orlov1}, $A$ is a geometric~noncommutative~$k$-scheme. 
\begin{example}[Dynkin quivers]
Let $\Delta$ be a Dynkin quiver of type $A_n$, $D_n$, $E_6$, $E_7$ or $E_8$. Recall that its Coxeter number $h$ is equal to $n+1$, $2(n-1)$, $12$, $18$ or $30$. It is well-known that the quiver $k$-algebra $A:=k\Delta$ has fractional CY-dimension $\frac{h-2}{h}$. Consequently, in all these cases the geometric noncommutative $k$-scheme $A$ is {\em not} Morita equivalent to a dg category of the form $\mathrm{perf}_\dg(Y)$, with $Y$ a smooth proper $k$-scheme.
\end{example}
Consider the largest semi-simple quotient $A/J$ of $A$, where $J$ stands for the Jacobson radical. Thanks to Artin-Wedderburn's theorem, $A/J$ is Morita equivalent to the product $D_1 \times \cdots \times D_n$, where $V_1, \ldots, V_n$ stand for the simple (right) $A/J$-modules and $D_1:=\mathrm{End}_{A/J}(V_1), \ldots, D_n:=\mathrm{End}_{A/J}(V_n)$ for the associated division $k$-algebras. Let us denote by $k_1, \ldots, k_n$ the centers of the division $k$-algebras $D_1, \ldots, D_n$.
\begin{theorem}\label{thm:finite1}
Assume that the quotient $k$-algebra $A/J$ is separable (this holds, for example, when $k$ is perfect). Under this assumption, we have the implication $\sum_{i=1}^n \mathrm{R}_0(\mathrm{Spec}(k_i)) \Rightarrow \mathrm{R}_{\mathrm{even}}(A)$. In particular, the conjecture $\mathrm{R}_{\mathrm{even}}(A)$ holds when $\mathrm{char}(k)>0$.
\end{theorem}
\begin{remark}[Artin $L$-functions]
Since $A/J$ is separable, the finite field extension $k_i/k$, with $1\leq i \leq n$, is also separable. Therefore, under the classical Galois-Grothendieck correspondence, the $k$-scheme $\mathrm{Spec}(k_i)$ corresponds to the finite set $\mathrm{Spec}(k_i)(\overline{k})$ equipped with the continuous action of the absolute Galois group $\mathrm{Gal}(\overline{k}/k)$. Consequently, the $L$-function $L_0(\mathrm{Spec}(k_i);s)$ (used in conjecture $\mathrm{R}_0(\mathrm{Spec}(k_i))$) reduces to the classical Artin $L$-function $L(\rho_i;s)$ associated to the $\bbC$-linear representation $\rho_i\colon \mathrm{Gal}(\overline{k}/k) \to \mathrm{GL}(\bbC^{\mathrm{Spec}(k_i)(\overline{k})})$.
\end{remark}
\subsection*{Finite-dimensional dg algebras} Let $A$ be a smooth {\em finite-dimensional} dg $k$-algebra in the sense of Orlov \cite{Orlov2}. As proved in \cite[Cor.~3.4]{Orlov2}, $A$ is a geometric noncommutative $k$-scheme. Following \cite[Def.~2.3]{Orlov2}, consider the quotient $A/J_+$, where $J_+$ stands for the external dg Jacobson radical of $A$.
\begin{theorem}\label{thm:finite2}
Assume that the quotient dg $k$-algebra $A/J_+$ is separable in the sense of \cite[Def.~2.11]{Orlov2} (this holds when $k$ is perfect) and that $\mathrm{char}(k)>0$. Under these assumptions, the conjecture $\mathrm{R}_{\mathrm{even}}(A)$ holds.
\end{theorem}
\begin{remark}
In the above Theorems \ref{thm:finite1} and \ref{thm:finite2}, the conjecture $\mathrm{R}_{\mathrm{odd}}(A)$ also holds; consult \S\ref{sec:proof} below.
\end{remark}

\section{Preliminaries}\label{sub:Orlov}
Let $k$ be a field. Throughout the note, we will assume some basic familiarity with the language of dg categories (consult Keller's survey \cite{Keller}) and will write $\dgcat(k)$ for the category of (small) dg categories.
\subsection{Geometric noncommutative schemes}\label{sub:geometric}
Following Orlov \cite[Def.~4.3]{Orlov1}, a dg category $\cA$ is called a {\em geometric noncommutative $k$-scheme} if there exists a smooth proper $k$-scheme $X$ and an admissible triangulated subcategory $\mathfrak{A}$ of $\perf(X)$ such that $\cA$ and the full dg subcategory $\mathfrak{A}_\dg$ of $\perf_\dg(X)$, consisting of the objects of $\mathfrak{A}$, are Morita equivalent. Every geometric noncommutative $k$-scheme $\cA$ is, in particular, a smooth proper dg category in the sense of Kontsevich\footnote{Orlov asked in \cite[Question~4.4]{Orlov1} if there exist smooth proper dg categories which are {\em not} geometric noncommutative schemes. To the best of the author's knowledge, this question remains wide open.} \cite{ENS}.
\begin{lemma}\label{lem:extension}
Let $k'/k$ be a field extension. If the dg category $\cA$ is a geometric noncommutative $k$-scheme, then the dg category $\cA\otimes_k k'$ is a geometric noncommutative $k'$-scheme.
\end{lemma}
\begin{proof}
By definition, there exists a smooth proper $k$-scheme $X$ and an admissible triangulated subcategory $\mathfrak{A}$ of $\perf(X)$ such that $\cA$ and $\mathfrak{A}_\dg$ are Morita equivalent. Consequently, we obtain an induced Morita equivalence between $\cA\otimes_k k'$ and $\mathfrak{A}_\dg \otimes_k k'$. Consider the following Morita equivalence:
\begin{eqnarray}\label{eq:canonical1}
\perf_\dg(X) \otimes_k k' \too \perf_\dg(X\times_k k') && \cF \mapsto \cF\times_k k'\,.
\end{eqnarray}
Let us denote by $\mathfrak{A}'$ the smallest full triangulated subcategory of $\perf(X\times_k k')$ containing the objects $\cF \times_k k'$, with $\cF \in \mathfrak{A}$, and by $\mathfrak{A}'_\dg$ the associated full dg subcategory of $\perf_\dg(X\times_kk')$. By construction, the above Morita equivalence \eqref{eq:canonical1} restricts to a Morita equivalence $\mathfrak{A}_\dg \otimes_k k' \to \mathfrak{A}'_\dg$. Therefore, the proof follows now from the fact that $\mathfrak{A}'$ is an admissible triangulated subcategory of $\perf(X\times_k k')$. 
\end{proof}
\subsection{Additive invariants}\label{sub:additive}
Recall from \cite[\S2.1]{book} that a functor $E\colon \dgcat(k) \to \mathrm{D}$, with values in an additive category, is called an {\em additive invariant} if it satisfies the following two conditions:
\begin{itemize}
\item[(i)] It sends Morita equivalences to isomorphisms.
\item[(ii)] Let $\cB, \cC \subseteq \cA$ be dg categories inducing a semi-orthogonal decompositions $H^0(\cA)=\langle H^0(\cB), H^0(\cC)\rangle$ in the sense of Bondal-Orlov \cite[Def.~2.4]{BO}. Under these notations, the inclusions $\cB \subseteq \cA$ and $\cC\subseteq \cA$ induce an isomorphism $E(\cB) \oplus E(\cC) \to E(\cA)$.
\end{itemize}
\begin{lemma}\label{lem:extension1}
Let $k'/k$ be a field extension. Given an additive invariant $E\colon \dgcat(k') \to \mathrm{D}$, the composed functor $E(-\otimes_k k')\colon \dgcat(k) \to \mathrm{D}$ is also an additive invariant.
\end{lemma}
\begin{proof}
Condition (i) follows from the fact that the functor $-\otimes_k k'$ preserves Morita equivalences; consult \cite[Prop.~7.1]{Artin}. Concerning condition (ii), let $\cB, \cC \subseteq \cA$ be dg categories inducing a semi-orthogonal decomposition $H^0(\cA)=\langle H^0(\cB), H^0(\cC)\rangle$. The associated dg categories $\mathrm{pre}(\cB\otimes_k k'), \mathrm{pre}(\cC\otimes_k k') \subseteq \mathrm{pre}(\cA\otimes_k k')$, where $\mathrm{pre}(-)$ stands for Bondal-Kapranov's pretriangulated envelope \cite{BK}, also induce a semi-orthogonal decomposition $H^0(\mathrm{pre}(\cA\otimes_k k'))= \langle \mathrm{pre}(\cB\otimes_k k'), \mathrm{pre}(\cC\otimes_k k')\rangle$. Therefore, since the canonical dg functors $\cA\otimes_k k' \to \mathrm{pre}(\cA\otimes_k k')$, $\cB\otimes_k k' \to \mathrm{pre}(\cB\otimes_k k')$, and $\cC\otimes_k k' \to \mathrm{pre}(\cC\otimes_k k')$, are Morita equivalences, the proof of condition (ii) follows now from the fact that the functor $E$ satisfies condition (ii).
\end{proof}
Consult \cite[\S2.3]{book} for the construction of the universal additive invariant $U\colon \dgcat(k) \to \mathrm{Hmo}_0(k)$. Given any additive invariant $E$, there exists a unique $\bbZ$-linear functor $\overline{E}$ making the following diagram commute:
\begin{eqnarray}\label{eq:factorization}
\xymatrix{
\dgcat(k) \ar[d]_-U \ar[rr]^-{E} && \mathrm{D} \\
\mathrm{Hmo}_0(k) \ar@/_/[urr]_-{\overline{E}} &&\,.
}
\end{eqnarray}
\section{Noncommutative $l$-adic cohomology}
Let $k$ be a field. Given a prime number $l\neq \mathrm{char}(k)$, recall from \cite[\S2.5]{Atiyah} the construction of the $l$-adic \'etale $K$-theory functor with values in the (homotopy) category of spectra:
\begin{eqnarray}\label{eq:functor-l}
K^{\mathrm{et}}(-)_{\widehat{l}}\colon \dgcat(k) \too \mathrm{Spt} && \cA \mapsto \mathrm{holim}_{n\geq 0}\, K^{\mathrm{et}}(\cA;\bbZ/l^n)\,.
\end{eqnarray}
By construction, the homotopy groups $\pi_\ast(K^{\mathrm{et}}(\cA)_{\widehat{l}})$ are modules over the ring of $l$-adic integers $\bbZ_l$.
\begin{definition}[Noncommutative $l$-adic cohomology]
Given a dg category $\cA$, its noncommutative $l$-adic cohomology is defined as follows ($\overline{k}$ stands for a fixed separable closure of $k$):
\begin{eqnarray}\label{eq:NC-homology}
\mathrm{H}_{\mathrm{even}, l}(\cA):= \pi_0(K^{\mathrm{et}}(\cA\otimes_k \overline{k})_{\widehat{l}})\otimes_\bbZ \bbZ[1/l] && \mathrm{H}_{\mathrm{odd},l}(\cA):= \pi_1(K^{\mathrm{et}}(\cA\otimes_k \overline{k})_{\widehat{l}})\otimes_\bbZ \bbZ[1/l] \,.
\end{eqnarray}
\end{definition}
Note that, by construction, the noncommutative $l$-adic cohomology groups \eqref{eq:NC-homology} are $\bbQ_l$-vector spaces. Moreover, they are equipped with a continuous action of the absolute Galois group $\mathrm{Gal}(\overline{k}/k)$. Consequently, we obtain the following well-defined functors with values in the category of $\bbQ_l$-linear $\mathrm{Gal}(\overline{k}/k)$-modules:
\begin{equation}\label{eq:l-adic}
\mathrm{H}_{\mathrm{even}, l}(-), \mathrm{H}_{\mathrm{odd}, l}(-)\colon \dgcat(k) \too \mathrm{Gal}(\overline{k}/k)\text{-}\mathrm{Mod}\,.
\end{equation}
\begin{proposition}\label{prop:additive}
The functors \eqref{eq:l-adic} are additive invariants.
\end{proposition}
\begin{proof}
The $l$-adic \'etale $K$-theory functor \eqref{eq:functor-l} is an additive invariant; consult \cite[\S2.5]{Atiyah}. Therefore, the proof follows from the above general Lemma \ref{lem:extension1}.
\end{proof}
\begin{proposition}\label{prop:etale}
Given a smooth proper $k$-scheme $X$, we have isomorphisms of $\mathrm{Gal}(\overline{k}/k)$-modules
$$
\mathrm{H}_{\mathrm{even}, l}(\perf_\dg(X)) \simeq \bigoplus_{w \, \mathrm{even}} H^w_{\mathrm{et}}(X\times_k \overline{k}; \bbQ_l(\frac{w}{2})) \quad \,\,\mathrm{H}_{\mathrm{odd}, l}(\perf_\dg(X)) \simeq \bigoplus_{w \, \mathrm{odd}} H^w_{\mathrm{et}}(X\times_k \overline{k}; \bbQ_l(\frac{w-1}{2}))\,,
$$
where $H_{\mathrm{et}}(-)$ stands for \'etale cohomology.
\end{proposition}
\begin{proof}
Since the $k$-scheme $X$ is smooth and proper, the associated $\overline{k}$-scheme $X\times_k \overline{k}$ is, in particular, regular and separated. These conditions imply that Thomason's \'etale descent spectral sequence \cite[Thm.~4.1]{Thomason-etale} is well-defined and degenerates rationally; consult Soul\'e \cite[\S3.3.2]{Soule}. Consequently, we obtain an isomorphism of $\mathrm{Gal}(\overline{k}/k)$-modules between $\pi_0(K^{\mathrm{et}}(\perf_\dg(X\times_k \overline{k}))_{\widehat{l}})\otimes_\bbZ \bbZ[1/l]$ and the direct sum $\bigoplus_{w \, \mathrm{even}} H^w_{\mathrm{et}}(X\times_k \overline{k}; \bbQ_l(\frac{w}{2}))$ and between $\pi_1(K^{\mathrm{et}}(\perf_\dg(X\times_k \overline{k}))_{\widehat{l}})\otimes_\bbZ \bbZ[1/l]$ and the direct sum $\bigoplus_{w \, \mathrm{odd}} H^w_{\mathrm{et}}(X\times_k \overline{k}; \bbQ_l(\frac{w-1}{2})$. The proof follows now from the Morita equivalence $\perf_\dg(X) \otimes_k \overline{k} \to \perf_\dg(X\times_k \overline{k}), \cF \mapsto \cF \times _k \overline{k}$.
\end{proof}
\begin{lemma}\label{lem:finite}
Given a geometric noncommutative $k$-scheme $\cA$, the associated $\bbQ_l$-linear $\mathrm{Gal}(\overline{k}/k)$-modules \eqref{eq:NC-homology} are finite-dimensional.
\end{lemma}
\begin{proof}
By definition, there exists a smooth proper $k$-scheme $X$ and an admissible triangulated subcategory $\mathfrak{A}$ of $\perf(X)$ such that $\cA$ and $\mathfrak{A}_\dg$ are Morita equivalent. Let us denote by ${}^\perp\mathfrak{A}$ the left orthogonal complement of $\mathfrak{A}$ in $\perf(X)$ and by ${}^\perp\mathfrak{A}_\dg$ the associated full dg subcategory of $\perf_\dg(X)$. By construction, we have a semi-orthogonal decomposition $H^0(\perf_\dg(X))=\langle H^0(\mathfrak{A}_\dg), H^0({}^\perp \mathfrak{A}_\dg)\rangle$, i.e., $\perf(X) = \langle \mathfrak{A}, {}^\perp\mathfrak{A} \rangle$. Therefore, making use of Proposition \ref{prop:additive}, we obtain the following computations:
\begin{eqnarray*}
\mathrm{H}_{\mathrm{even},l}(\perf_\dg(X))\simeq \mathrm{H}_{\mathrm{even},l}(\cA) \oplus \mathrm{H}_{\mathrm{even}, l}({}^\perp\mathfrak{A}_\dg) && \mathrm{H}_{\mathrm{odd},l}(\perf_\dg(X))\simeq \mathrm{H}_{\mathrm{odd},l}(\cA) \oplus \mathrm{H}_{\mathrm{odd}, l}({}^\perp\mathfrak{A}_\dg)\,.
\end{eqnarray*}
The proof follows now from Proposition \ref{prop:etale} and from the fact that the \'etale cohomology $\bbQ_l$-vector spaces $\{H^w_{\mathrm{et}}(X\times_k \overline{k};\bbQ_l(\frac{w}{2}))\}_{w\,\mathrm{even}}$ and $\{H^w_{\mathrm{et}}(X\times_k \overline{k};\bbQ_l(\frac{w-1}{2}))\}_{w\,\mathrm{odd}}$ are finite-dimensional.
\end{proof}
\section{Noncommutative $L$-functions}\label{sec:NC-L}
Let $k$ be a global field. We start by fixing some important notations:
\begin{notation}\label{not:key1}
Given a non-archimedean place $\nu \in \Sigma_k$, let us write $k_\nu$ for the completion of $k$ at $\nu$, $\cO_\nu$ for the valuation ring of $k_\nu$, $\kappa_\nu$ for the residue field of $\cO_\nu$, $p_\nu$ for the characteristic of $\kappa_\nu$, $N_\nu$ for the cardinality of the finite field $\kappa_\nu$, $I_\nu$ for the {\em inertia subgroup}, i.e., the kernel of the canonical surjective map $\mathrm{Gal}(\overline{k_\nu}/k_\nu) \twoheadrightarrow \mathrm{Gal}(\overline{\kappa_\nu}/\kappa_\nu)$, and $\pi_\nu \in \mathrm{Gal}(\overline{\kappa_\nu}/\kappa_\nu)\simeq \widehat{\bbZ}$ for the {\em geometric Frobenius}, i.e., the inverse of the arithmetic Frobenius $\lambda \mapsto \lambda^{N_\nu}$.
\end{notation}
\begin{notation}\label{not:key2}
Let us choose a prime number $l_\nu \neq p_\nu$ and a field embedding $\iota_\nu\colon \bbQ_{l_\nu}\hookrightarrow \bbC$ for every $\nu \in \Sigma_k$.
\end{notation}
The above Lemmas \ref{lem:extension} and \ref{lem:finite} enable the following definition:
\begin{definition}[Noncommutative $L$-functions]\label{def:L-functions}
 Given a geometric noncommutative $k$-scheme $\cA$, its noncommutative $L$-functions are defined as follows (we are implicitly using the field embedding $\iota_\nu$):
$$
L_{\mathrm{even}}(\cA;s):=\prod_{\nu \in \Sigma_k} L_{\mathrm{even}, \nu}(\cA;s) \quad \quad L_{\mathrm{even}, \nu}(\cA;s):=\frac{1}{\mathrm{det}\big(\mathrm{id} - N_\nu^{-s}(\pi_\nu \otimes_{\bbQ_{l_\nu}} \bbC)\,|\,\mathrm{H}_{\mathrm{even}, l_\nu}(\cA\otimes_k k_\nu)^{I_\nu} \otimes_{\bbQ_{l_\nu}} \bbC\big)}
$$
$$
L_{\mathrm{odd}}(\cA;s):=\prod_{\nu \in \Sigma_k} L_{\mathrm{odd}, \nu}(\cA;s) \quad \quad L_{\mathrm{odd}, \nu}(\cA;s):=\frac{1}{\mathrm{det}\big(\mathrm{id} - N_\nu^{-s}(\pi_\nu \otimes_{\bbQ_{l_\nu}} \bbC)\,|\,\mathrm{H}_{\mathrm{odd}, l_\nu}(\cA\otimes_k k_\nu)^{I_\nu} \otimes_{\bbQ_{l_\nu}} \bbC\big)}\,.
$$
\end{definition}
\begin{proposition}\label{prop:product}
Let $\cB, \cC \subseteq \cA$ be geometric noncommutative $k$-schemes inducing a semi-orthogonal decomposition $H^0(\cA)=\langle H^0(\cB), H^0(\cC)\rangle$. Under these notations, we have the following equalities:
\begin{eqnarray}\label{eq:equalities-semi}
L_{\mathrm{even}}(\cA;s) = L_{\mathrm{even}}(\cB;s) \cdot L_{\mathrm{even}}(\cC;s) && L_{\mathrm{odd}}(\cA;s) = L_{\mathrm{odd}}(\cB;s) \cdot L_{\mathrm{odd}}(\cC;s)\,.
\end{eqnarray}
\end{proposition}
\begin{proof}
Let $\nu \in \Sigma_k$ be a non-archimedean place. Thanks to Proposition \ref{prop:additive} and Lemma \ref{lem:extension1}, we have an isomorphism of $\mathrm{Gal}(\overline{k_\nu}/k_\nu)$-modules between $\mathrm{H}_{\mathrm{even}, l_\nu} (\cA\otimes_k k_\nu)$ and $\mathrm{H}_{\mathrm{even}, l_\nu} (\cB\otimes_k k_\nu) \oplus \mathrm{H}_{\mathrm{even}, l_\nu} (\cC\otimes_k k_\nu)$. This implies that $L_{\mathrm{even},\nu}(\cA;s)= L_{\mathrm{even},\nu}(\cB;s) \cdot L_{\mathrm{even},\nu}(\cC;s)$. Consequently, the left-hand side of \eqref{eq:equalities-semi} follows from Definition \ref{def:L-functions}. The proof of the odd case is similar: simply replace the word ``even'' by~the~word~``odd''.
\end{proof}
\begin{proposition}\label{prop:equalities}
Given a smooth proper $k$-scheme $X$, we have the following equalities:
\begin{eqnarray}\label{eq:equalities-key}
L_{\mathrm{even}}(\perf_\dg(X);s) = \prod_{w\, \mathrm{even}}L_w(X;s+\frac{w}{2}) && L_{\mathrm{odd}}(\perf_\dg(X);s) = \prod_{w\, \mathrm{odd}}L_w(X;s+\frac{w-1}{2}) \label{eq:eq-L} \,.
\end{eqnarray}
\end{proposition}
Roughly speaking, Proposition \ref{prop:equalities} shows that the noncommutative even/odd $L$-function of $\perf_\dg(X)$ may be understood as the ``weight normalization'' of the product of the $L$-functions of $X$ of even/odd weight.
\begin{proof}
Recall first that the $L$-function of $X$ of weight $w$ is defined as follows (consult Notations \ref{not:key1}-\ref{not:key2}):
$$L_w(X;s):=\prod_{\nu \in \Sigma_k} L_{w,\nu}(X;s) \quad \quad L_{w,\nu}(X;s) := \frac{1}{\mathrm{det}(\id - N_\nu^{-s}(\pi_\nu \otimes_{\bbQ_{l_\nu}}\bbC)\,|\,H^w_{\mathrm{et}}((X\times_k k_\nu)\times_{k_\nu} \overline{k_\nu})^{I_\nu}\otimes_{\bbQ_{l_\nu}}\bbC)}\,.$$
Note that we have the following isomorphisms of $\mathrm{Gal}(\overline{\kappa_\nu}/\kappa_\nu)$-modules
\begin{eqnarray}
\mathrm{H}_{\mathrm{even}, l_\nu}(\perf_\dg(X)\otimes_k k_\nu)^{I_\nu} & \simeq & \mathrm{H}_{\mathrm{even}, l_\nu}(\perf_\dg(X\times_k k_\nu))^{I_\nu} \label{eq:star1} \\
&\simeq & \big(\bigoplus_{w\,\mathrm{even}}H^w_{\mathrm{et}}((X\times_k k_\nu)\times_{k_\nu}\overline{k_\nu}; \bbQ_{l_\nu}(\frac{w}{2}))\big)^{I_\nu} \label{eq:star2} \\
& \simeq &  \bigoplus_{w\,\mathrm{even}}H^w_{\mathrm{et}}((X\times_k k_\nu)\times_{k_\nu}\overline{k_\nu}; \bbQ_{l_\nu}(\frac{w}{2}))^{I_\nu} \nonumber \\
& \simeq & \bigoplus_{w\,\mathrm{even}}\big(H^w_{\mathrm{et}}((X\times_k k_\nu)\times_{k_\nu}\overline{k_\nu}; \bbQ_{l_\nu})^{I_\nu}\otimes_{\bbQ_{l_\nu}} \bbQ_{l_\nu}(\frac{w}{2})\big)\,, \label{eq:star3}
\end{eqnarray}
where \eqref{eq:star1} follows from the Morita equivalence $\perf_\dg(X) \otimes_k k_\nu \to \perf_\dg(X\times_k k_\nu), \cF \mapsto \cF\times_k k_\nu$, \eqref{eq:star2} from Proposition \ref{prop:etale}, and \eqref{eq:star3} from the fact that the $\mathrm{Gal}(\overline{k_\nu}/k_\nu)$-module $\bbQ_{l_\nu}(\frac{w}{2})$ is unramified (i.e., $I_\nu$ acts trivially). Moreover, since the action of the geometric Frobenius on $\bbQ_{l_\nu}(\frac{w}{2})$ is given by multiplication by $N_\nu^{-\frac{w}{2}}$, the $\mathrm{Gal}(\overline{\kappa_\nu}/\kappa_\nu)$-module \eqref{eq:star3} may be identified with $\bigoplus_{w\,\mathrm{even}} H^w_{\mathrm{et}}((X\times_k k_\nu)\times_{k_\nu}\overline{k_\nu}; \bbQ_{l_\nu})^{I_\nu}$ where the geometric Frobenius acts diagonally as $\bigoplus_{w\,\mathrm{even}} \pi_\nu \cdot N_\nu^{-\frac{w}{2}}$. This implies the following equalities:
\begin{eqnarray*}
L_{\mathrm{even}, \nu}(\perf_\dg(X);s) & = & \prod_{w \,\mathrm{even}} \frac{1}{\mathrm{det}(\id - N_\nu^{-(s+\frac{w}{2})}(\pi_\nu \otimes_{\bbQ_{l_\nu}}\bbC)\,|\,H^w_{\mathrm{et}}((X\times_k k_\nu)\times_{k_\nu} \overline{k_\nu};\bbQ_{l_\nu})^{I_\nu} \otimes_{\bbQ_{l_\nu}}\bbC)} \\
& = &  \prod_{w\, \mathrm{even}} L_{w,\nu}(X;s+\frac{w}{2})\,.
\end{eqnarray*}
Consequently, the left-hand side of \eqref{eq:equalities-key} follows now from the following equalities:
$$
L_{\mathrm{even}}(\perf_\dg(X);s) :=  \prod_{\nu \in \Sigma_k} L_{\mathrm{even},\nu}(\perf_\dg(X);s) =  \prod_{\nu \in \Sigma_k} \prod_{w\, \mathrm{even}} L_{w,\nu}(X;s+ \frac{w}{2}) = \prod_{w\, \mathrm{even}} L_{w}(X;s+ \frac{w}{2})\,.
$$
The proof of the odd case is similar: simply replace the word ``even'' by the word ``odd'' and $\frac{w}{2}$ by $\frac{w-1}{2}$. 
\end{proof}
\section{Proof of Theorem \ref{thm:convergence}}
Recall that $k$ is a finite field extension of $\bbQ$ (when $\mathrm{char}(k)=0$) or a finite field extension of $\bbF_q(t)$ (when $\mathrm{char}(k)>0$), where $\bbF_q$ is the finite field with $q$ elements. In what follows, we will write $\cO_k \subset k$ for the integral closure of $\bbZ$ in $k$ (when $\mathrm{char}(k)=0$) or for the integral closure of $\bbF_q[t]$ in $k$ (when $\mathrm{char}(k)>0$). 

Recall that since $\cA$ is a geometric noncommutative $k$-scheme, there exists a smooth proper $k$-scheme $X$ and an admissible triangulated subcategory $\mathfrak{A}$ of $\perf(X)$ such that $\cA$ and $\mathfrak{A}_\dg$ are Morita equivalent.
\begin{notation}
Given an integer $0 \leq w \leq 2\mathrm{dim}(X)$ and a prime number $l\neq \mathrm{char}(k)$, let us write $\beta_w$ for the dimension of the $\bbQ_l$-vector space $H^w_{\mathrm{et}}(X\times_k \overline{k}; \bbQ_l)$; it is well-known that this dimension is independent of $l$.
\end{notation}
\begin{lemma}\label{lem:equalities}
For every non-archimedean place $\nu \in \Sigma_k$, we have the equalities (consult Notation \ref{not:key2}):
\begin{eqnarray}\label{eq:equalities}
\mathrm{dim}_{\bbQ_{l_\nu}}\mathrm{H}_{\mathrm{even}, l_\nu}(\perf_\dg(X) \otimes_k k_\nu)= \sum_{w\,\mathrm{even}}\beta_w && \mathrm{dim}_{\bbQ_{l_\nu}}\mathrm{H}_{\mathrm{odd}, l_\nu}(\perf_\dg(X) \otimes_k k_\nu)= \sum_{w\,\mathrm{odd}}\beta_w\,.
\end{eqnarray}
\end{lemma}
\begin{proof}
We have the following isomorphisms of $\bbQ_{l_\nu}$-vector spaces
\begin{eqnarray}
\mathrm{H}_{\mathrm{even}, l_\nu}(\perf_\dg(X)\otimes_k k_\nu) & \simeq & \mathrm{H}_{\mathrm{even}, l_\nu}(\perf_\dg(X\times_k k_\nu)) \label{eq:star111} \\
&\simeq & \bigoplus_{w\,\mathrm{even}}H^w_{\mathrm{et}}((X\times_k k_\nu)\times_{k_\nu}\overline{k_\nu}; \bbQ_{l_\nu}(\frac{w}{2})) \label{eq:star222}\\
& \simeq &  \bigoplus_{w\,\mathrm{even}}H^w_{\mathrm{et}}((X\times_k k_\nu)\times_{k_\nu}\overline{k_\nu}; \bbQ_{l_\nu}) \nonumber \,,
\end{eqnarray}
where \eqref{eq:star111} follows from the Morita equivalence $\perf_\dg(X) \otimes_k k_\nu \to \perf_\dg(X\times_k k_\nu), \cF \mapsto \cF\times_k k_\nu$, and \eqref{eq:star222} from Proposition \ref{prop:etale}. Consequently, the left-hand side of \eqref{eq:equalities} follows from the (well-known) fact that the canonical homomorphism $H^w_{\mathrm{et}}(X\times_k \overline{k}; \bbQ_{l_\nu}) \to H^w_{\mathrm{et}}((X\times_k k_\nu)\times_{k_\nu} \overline{k_\nu}; \bbQ_{l_\nu})$ is invertible. The proof of the odd case is similar: simply replace the word ``even'' by the word ``odd'' and $\frac{w}{2}$ by $\frac{w-1}{2}$.
\end{proof}
\begin{notation}
Let $S_X$ be the (finite) set of prime ideals of $\cO_k$ where $X$ has bad reduction, $\cO_k[S_X^{-1}]$ the localized ring, and $\Sigma_X$ the subset of $\Sigma_k$ corresponding to the prime ideals of $\cO_k[S_X^{-1}]$. Given a prime ideal $\cP\triangleleft \cO_k[S_X^{-1}]$, we will denote by $\nu_{\cP} \in \Sigma_X$ the corresponding non-archimedean place. Similarly, given a non-archimedean place $\nu \in \Sigma_X$, we will denote by $\cP_\nu \triangleleft \cO_k[S_X^{-1}]$ the corresponding prime ideal.
\end{notation}
\begin{lemma}\label{lem:eigenvalue}
Let $\nu \in \Sigma_X$ be a non-archimedean place and $\lambda$ an eigenvalue of the automorphism $\pi_\nu \otimes_{\bbQ_{l_\nu}} \bbC$ of the $\bbC$-vector space $\mathrm{H}_{\mathrm{even}, l_\nu}(\perf_\dg(X) \otimes_k k_\nu)^{I_\nu}\otimes_{\bbQ_{l_\nu}}\bbC$, resp. $\mathrm{H}_{\mathrm{odd}, l_\nu}(\perf_\dg(X) \otimes_k k_\nu)^{I_\nu}\otimes_{\bbQ_{l_\nu}}\bbC$. Under these notations, we have $|\lambda|=1$, resp. $|\lambda|=N_\nu^{\frac{1}{2}}$.
\end{lemma}
\begin{proof}
As explained in the proof of Proposition \ref{prop:equalities}, we have an isomorphism of $\mathrm{Gal}(\overline{\kappa_\nu}/\kappa_\nu)$-modules
\begin{eqnarray}\label{eq:isomorphism}
\mathrm{H}_{\mathrm{even},l_\nu}(\perf_\dg(X) \otimes_k k_\nu)^{I_\nu} \otimes_{\bbQ_{l_\nu}}\bbC & \simeq & \bigoplus_{w\,\mathrm{even}} H^w_{\mathrm{et}}((X\times_k k_\nu)\times_{k_\nu}\overline{k_\nu})^{I_\nu} \otimes_{\bbQ_{l_\nu}}\bbC\,,
\end{eqnarray}
where the automorphism $\pi_\nu \otimes_{\bbQ_{l_\nu}}\bbC$ on the left-hand side of \eqref{eq:isomorphism} corresponds to the diagonal automorphism $\bigoplus_{w\,\mathrm{even}}(\pi_\nu \otimes_{\bbQ_{l_\nu}}\bbC)\cdot N_\nu^{-\frac{w}{2}}$ on the right-hand side. Let $\lambda'$ be an eigenvalue of the automorphism $\pi_\nu \otimes_{\bbQ_{l_\nu}}\bbC$ of the $\bbC$-vector space $H^w_{\mathrm{et}}((X\times_k k_\nu)\times_{k_\nu}\overline{k_\nu})^{I_\nu} \otimes_{\bbQ_{l_\nu}}\bbC$, with $0 \leq w \leq 2\mathrm{dim}(X)$. Since $X$ has good reduction at the prime ideal $\cP_\nu$, it follows from Deligne's proof of the Weil conjecture (consult \cite{Weil1}) and from the smooth proper base-change property of \'etale cohomology that $|\lambda'|=N_\nu^{\frac{w}{2}}$. Consequently, we conclude that $|\lambda|=N_\nu^{\frac{w}{2}}N_\nu^{-\frac{w}{2}}=1$. The proof of the odd case is similar: simply replace the word ``even'' by the word ``odd'' and the diagonal automorphism $\bigoplus_{w\,\mathrm{even}}(\pi_\nu \otimes_{\bbQ_{l_\nu}}\bbC)\cdot N_\nu^{-\frac{w}{2}}$ by $\bigoplus_{w\,\mathrm{odd}}(\pi_\nu \otimes_{\bbQ_{l_\nu}}\bbC)\cdot N_\nu^{-\frac{w-1}{2}}$.
\end{proof}
Now, note that since the set $\Sigma_k \backslash \Sigma_X$ is finite, in order to prove Theorem \ref{thm:convergence}, we can (and will) replace the infinite products $\prod_{\nu \in \Sigma_k} L_{\mathrm{even}, \nu}(\cA;s)$ and $\prod_{\nu \in \Sigma_k} L_{\mathrm{odd}, \nu}(\cA;s)$ by the infinite products $\prod_{\nu \in \Sigma_X} L_{\mathrm{even}, \nu}(\cA;s)$ and $\prod_{\nu \in \Sigma_X} L_{\mathrm{odd}, \nu}(\cA;s)$, respectively.
\begin{notation}
Given a non-archimedean place $\nu \in \Sigma_k$ and an integer $n\geq 1$, consider the complex numbers:
\begin{eqnarray*}
\#_{(+,\nu, n)}& := & \mathrm{trace}\big((\pi_\nu \otimes_{\bbQ_{l_\nu}} \bbC)^{\circ n}\,|\,\mathrm{H}_{\mathrm{even}, l_\nu}(\cA\otimes_k k_\nu)^{I_\nu} \otimes_{\bbQ_{l_\nu}} \bbC\big) \\
 \#_{(-,\nu, n)} & := & \mathrm{trace}\big((\pi_\nu \otimes_{\bbQ_{l_\nu}} \bbC)^{\circ n}\,|\,\mathrm{H}_{\mathrm{odd}, l_\nu}(\cA\otimes_k k_\nu)^{I_\nu} \otimes_{\bbQ_{l_\nu}} \bbC\big) \,.
\end{eqnarray*}
\end{notation}
\begin{proposition}\label{prop:inequalities}
Given a non-archimedean place $\nu \in \Sigma_X$ and an integer $n\geq 1$, we have the inequalities:
\begin{eqnarray*}
|\#_{(+,\nu, n)}|\leq \sum_{w\, \mathrm{even}} \beta_w && |\#_{(-,\nu, n)}|\leq (\sum_{w\, \mathrm{odd}} \beta_w)\cdot N_\nu^{\frac{n}{2}}\,.
\end{eqnarray*}
\end{proposition}
\begin{proof}
Let us write $\chi_{(+,\nu)}$, resp. $\chi_{(-,\nu)}$, for the dimension of the $\bbC$-vector space 
\begin{eqnarray}\label{eq:cohomologies}
\mathrm{H}_{\mathrm{even}, l_\nu}(\cA\otimes_k k_\nu)^{I_\nu}\otimes_{\bbQ_{l_\nu}}\bbC\,, && \mathrm{resp.}\,\,\,\,\, \mathrm{H}_{\mathrm{odd}, l_\nu}(\cA\otimes_k k_\nu)^{I_\nu}\otimes_{\bbQ_{l_\nu}}\bbC\,,
\end{eqnarray}
and $\{\lambda_{(+,\nu, 1)}, \ldots, \lambda_{(+,\nu, \chi_{(+,\nu)})}\}$, resp. $\{\lambda_{(-,\nu, 1)}, \ldots, \lambda_{(-,\nu, \chi_{(-,\nu)})}\}$, for the set of eigenvalues (with multiplicities) of the automorphism $\pi_\nu \otimes_{\bbQ_{l_\nu}}\bbC$ of \eqref{eq:cohomologies}. Moreover, let us write ${}^\perp\mathfrak{A}$ for the left orthogonal complement of $\mathfrak{A}$ in $\perf(X)$ and ${}^\perp\mathfrak{A}_\dg$ for the associated full dg subcategory of $\perf_\dg(X)$. By construction, we have a semi-orthogonal decomposition $H^0(\perf_\dg(X))=\langle H^0(\mathfrak{A}_\dg), H^0({}^\perp \mathfrak{A}_\dg)\rangle$, i.e., $\perf(X) = \langle \mathfrak{A}, {}^\perp\mathfrak{A} \rangle$. Therefore, by combining Proposition \ref{prop:additive} with Lemma \ref{lem:extension1}, we obtain an isomorphism of $\mathrm{Gal}(\overline{k_\nu}/k_\nu)$-modules between $\mathrm{H}_{\mathrm{even},l_\nu}(\perf_\dg(X)\otimes_k k_\nu)$ and the direct sum $\mathrm{H}_{\mathrm{even},l_\nu}(\cA\otimes_k k_\nu) \oplus \mathrm{H}_{\mathrm{even}, l_\nu}({}^\perp\mathfrak{A}_\dg\otimes_k k_\nu)$ and between $\mathrm{H}_{\mathrm{odd},l_\nu}(\perf_\dg(X)\otimes_k k_\nu)$ and the direct sum $\mathrm{H}_{\mathrm{odd},l_\nu}(\cA\otimes_k k_\nu) \oplus \mathrm{H}_{\mathrm{odd}, l_\nu}({}^\perp\mathfrak{A}_\dg\otimes_k k_\nu)$.
On the one hand, thanks to Lemma \ref{lem:eigenvalue}, this implies that if $\lambda$ is an eigenvalue of the automorphism $\pi_\nu \otimes_{\bbQ_{l_\nu}}\bbC$ of \eqref{eq:cohomologies}, then  $|\lambda|=1$, resp. $|\lambda|=N_\nu^{\frac{1}{2}}$. On the other hand, thanks to Lemma \ref{lem:equalities}, this implies that $\chi_{(+,\nu)}\leq \sum_{w\,\mathrm{even}}\beta_w$, resp. $\chi_{(-,\nu)}\leq \sum_{w\,\mathrm{odd}}\beta_w$. As a consequence, we obtain the following inequalities:
$$
|\#_{(+, \nu, n)}| = |\lambda^n_{(+,\nu,1)}+ \cdots + \lambda^n_{(+, \nu, \chi_{(+,\nu)})}| \leq |\lambda_{(+,\nu,1)}|^n+ \cdots + |\lambda_{(+, \nu, \chi_{(+,\nu)})}|^n  = \chi_{(+, \nu)} \leq \sum_{w\,\mathrm{even}}\beta_w
$$
$$
|\#_{(-, \nu, n)}| = |\lambda^n_{(-,\nu,1)}+ \cdots + \lambda^n_{(-, \nu, \chi_{(+,\nu)})}| \leq |\lambda_{(-,\nu,1)}|^n+ \cdots + |\lambda_{(-, \nu, \chi_{(+,\nu)})}|^n  = \chi_{(-, \nu)}\cdot N_\nu^{\frac{n}{2}} \leq (\sum_{w\,\mathrm{odd}}\beta_w) \cdot N_\nu^{\frac{n}{2}}\,.
$$
This concludes the proof of Proposition \ref{prop:inequalities}.
\end{proof}
The following general result, whose proof is a simple linear algebra exercise, is well-known:
\begin{lemma}\label{lem:exercise1}
Given an endomorphism $f\colon V \to V$ of a finite-dimensional $\bbC$-linear vector space, we have the equality of formal power series $\mathrm{log}(\frac{1}{\mathrm{det}(\id - tf|V)})= \sum_{n\geq 1} \mathrm{trace}(f^{\circ n})\frac{t^n}{n}$, where $\mathrm{log}(t):= \sum_{n\geq1} \frac{(-1)^{n+1}}{n}(t-1)^n$.
\end{lemma}
Given a non-archimeadean place $\nu \in \Sigma_X$, consider the formal power series and their exponentiations:
\begin{eqnarray*}
\phi_{(+, \nu)}(t) := \sum_{n \geq 1} \#_{(+, \nu, n)}\frac{t^n}{n} && \varphi_{(+,\nu)}(t):=\mathrm{exp}(\phi_{(+,\nu)}(t))= \sum_{n\geq 0} a_{(+, \nu, n)}t^n \\
\phi_{(-, \nu)}(t) := \sum_{n \geq 1} \#_{(-, \nu, n)}\frac{t^n}{n} && \varphi_{(-,\nu)}(t):=\mathrm{exp}(\phi_{(-,\nu)}(t))= \sum_{n\geq 0} a_{(-, \nu, n)}t^n\,.
\end{eqnarray*}
Note that, thanks to Lemma \ref{lem:exercise1}, we have $\varphi_{(+,\nu)}(N_\nu^{-s})=L_{\mathrm{even},\nu}(\cA;s)$ and $\varphi_{(-,\nu)}(N_\nu^{-s})=L_{\mathrm{odd},\nu}(\cA;s)$ for every non-archimedean place $\nu \in \Sigma_X$.
\begin{definition}
Consider the following multiplicative Dirichlet series 
\begin{eqnarray*}
\varphi_+(s):= \sum_I \frac{b_{(+, I)}}{N(I)^s} &&  \varphi_-(s):= \sum_I \frac{b_{(-, I)}}{N(I)^s}\,,
\end{eqnarray*}
where $I\triangleleft \cO_k[S_X^{-1}]$ is an ideal, $b_{(+,I)}:=a_{(+, \nu_{\cP_1}, r_1)} \cdots a_{(+, \nu_{\cP_m}, r_m)}$ and $b_{(-,I)}:=a_{(-, \nu_{\cP_1}, r_1)} \cdots a_{(-, \nu_{\cP_m}, r_m)}$ are the products associated to the (unique) prime decomposition $I=\cP_1^{r_1} \cdots \cP_m^{r_m}$, and $N(I)$ is the norm of $I$.
\end{definition}
The following general result, concerning the absolute convergence of Dirichlet series, is well-known:
\begin{lemma}\label{lem:general}
We have the following equivalences
\begin{eqnarray}\label{eq:general}
\sum_I \frac{|b_{(+,I)}|}{|N(I)^s|} < \infty \Leftrightarrow \prod_\cP \sum_{n\geq 0} \frac{|b_{(+,\cP^n)}|}{|N(\cP^n)^s|} <\infty && \sum_I \frac{|b_{(-,I)}|}{|N(I)^s|} < \infty \Leftrightarrow \prod_\cP \sum_{n\geq 0} \frac{|b_{(-,\cP^n)}|}{|N(\cP^n)^s|} < \infty\,,
\end{eqnarray}
where $\cP \triangleleft\cO_k[S_X^{-1}]$ is a prime ideal. Moreover, if the left-hand side, resp. right-hand side, of \eqref{eq:general} converges, then we obtain the equality $\sum_I \frac{b_{(+,I)}}{N(I)^s} = \prod_\cP \sum_{n\geq 0} \frac{b_{(+,\cP^n)}}{N(\cP^n)^s}$, resp. $\sum_I \frac{b_{(-,I)}}{N(I)^s} = \prod_\cP \sum_{n\geq 0} \frac{b_{(-,\cP^n)}}{N(\cP^n)^s}$.
\end{lemma}
Given a prime ideal $\cP \triangleleft \cO_k[S_X^{-1}]$, recall that its norm $N(\cP)$ is defined as the cardinality of the quotient field $\cO_k[S_X^{-1}]/\cP$. Since $\cO_k[S_X^{-1}]/\cP$ is isomorphic to the residue field $\kappa_{\nu_\cP}$, we hence obtain the formal equalities:
$$\sum_{n \geq 0} \frac{b_{(+, \cP^n)}}{N(\cP^n)^s} = \sum_{n \geq 0} \frac{a_{(+, \nu_\cP, n)}}{N(\cP)^{ns}} = \sum_{n \geq 0} \frac{a_{(+, \nu_\cP, n)}}{N_{\nu_\cP}^{ns}} = \varphi_{(+,\nu_\cP)}(N_{\nu_\cP}^{-s}) = L_{\mathrm{even}, \nu_\cP}(\cA;s)$$
$$\sum_{n \geq 0} \frac{b_{(-, \cP^n)}}{N(\cP^n)^s} = \sum_{n \geq 0} \frac{a_{(-, \nu_\cP, n)}}{N(\cP)^{ns}} = \sum_{n \geq 0} \frac{a_{(-, \nu_\cP, n)}}{N_{\nu_\cP}^{ns}} = \varphi_{(-,\nu_\cP)}(N_{\nu_\cP}^{-s}) = L_{\mathrm{odd}, \nu_\cP}(\cA;s)\,.$$
Therefore, thanks to Lemma \ref{lem:general} and to classical properties of Dirichlet series, in order to prove Theorem \ref{thm:convergence}, it suffices to show the following claim: we have $\prod_\cP\sum_{n\geq 0} \frac{|a_{(+,\nu_\cP,n)}|}{N(\cP)^{nz}}<\infty$, resp. $\prod_\cP\sum_{n\geq 0} \frac{|a_{(-,\nu_\cP,n)}|}{N(\cP)^{nz}}<\infty$, for every real number $z>1$, resp. $z> \frac{3}{2}$.
Note that, by construction, we have the following inequalities:
\begin{eqnarray}\label{eq:inequalities-verylast}
\sum_{n \geq 0} \frac{|a_{(+, \nu_\cP, n)}|}{N(\cP)^{nz}} \leq \mathrm{exp}\big(\sum_{n \geq 1} \frac{|\#_{(+, \nu_\cP, n)}|}{nN(\cP)^{nz}}\big) &&\sum_{n \geq 0} \frac{|a_{(-, \nu_\cP, n)}|}{N(\cP)^{nz}} \leq \mathrm{exp}\big(\sum_{n \geq 1} \frac{|\#_{(-, \nu_\cP, n)}|}{nN(\cP)^{nz}}\big)\,.
\end{eqnarray}
Moreover, by taking $\mathrm{exp}(-)$ to Lemma \ref{lem:key} below, we observe that $\prod_\cP\mathrm{exp}(\sum_{n\geq 1} \frac{|\#_{(+,\nu_\cP,n)}|}{nN(\cP)^{nz}})<\infty$, resp. $\prod_\cP\mathrm{exp}(\sum_{n\geq 1} \frac{|\#_{(-,\nu_\cP,n)}|}{nN(\cP)^{nz}})<\infty$, for every real number $z>1$, resp. $z > \frac{3}{2}$. Consequently, by taking the product, over all the prime ideals $\cP\triangleleft\cO_k[S_X^{-1}]$, of the inequalities \eqref{eq:inequalities-verylast}, we obtain the aforementioned claim.
\begin{lemma}\label{lem:key}
We have $\sum_\cP\sum_{n\geq 1} \frac{|\#_{(+,\nu_\cP,n)}|}{nN(\cP)^{nz}}<\infty$, resp. $\sum_\cP\sum_{n\geq 1} \frac{|\#_{(-,\nu_\cP,n)}|}{nN(\cP)^{nz}}<\infty$, for every real number $z >1 $, resp. $z > \frac{3}{2}$.
\end{lemma}
\begin{proof}
Note that, thanks to Proposition \ref{prop:inequalities} and to the equality $N(\cP)=N_{\nu_\cP}$, it suffices to show that $\sum_\cP\sum_{n\geq 1} \frac{1}{N(\cP)^{nz}}<\infty$, resp. $\sum_\cP\sum_{n\geq 1} \frac{1}{N(\cP)^{nz}}<\infty$, for every real number $z >1$, resp. $z> \frac{3}{2}$.
Let us assume first that $\mathrm{char}(k)=0$. Recall that in this case $k$ is a finite field extension of $\bbQ$. In what concerns $\sum_\cP\sum_{n\geq 1} \frac{1}{N(\cP)^{nz}}<\infty$, with $z>1$, we have the following (in)equalities ($p\in \bbZ$ stands for a prime number)
\begin{eqnarray}
\sum_\cP\sum_{n\geq 1} \frac{1}{N(\cP)^{nz}} & \leq & \sum_p \sum_{\cP\mid p}\sum_{n\geq 1} \frac{1}{p^{nz}} \label{eq:star-1-1}\\
& \leq & [k:\bbQ] \cdot \sum_p\sum_{n\geq 1} \frac{1}{p^{nz}} \label{eq:star-2-2} \\
& = & [k:\bbQ] \cdot \sum_p \frac{1}{p^z-1} \label{eq:star-2-2-1} \\
& \leq & 2\cdot  [k:\bbQ] \cdot \sum_p \frac{1}{p^z} \label{eq:star-3-3} \\
& \leq & 2 \cdot [k:\bbQ] \cdot \sum_{n\geq 1} \frac{1}{n^z} < \infty \label{eq:star-4-4}\,,
\end{eqnarray}
where \eqref{eq:star-1-1} follows from the fact that the norm $N(\cP)$ is a power of $p$ whenever $\cP$ divides $p$, \eqref{eq:star-2-2} from the fact that the number of prime ideals $\cP$ which divide $p$ is always bounded by the degree $[k:\bbQ]$, \eqref{eq:star-2-2-1} from the (convergent) geometric series $\sum_{n \geq 0} \frac{1}{(p^z)^n}=\frac{1}{1- \frac{1}{p^z}}$, \eqref{eq:star-3-3} from a simple inspection, and \eqref{eq:star-4-4} from the fact that the Riemann zeta function $\zeta(s):=\sum_{n\geq 1} \frac{1}{n^s}$ is convergent when $\mathrm{Re}(s)>1$. Similarly, in what concerns $\sum_\cP\sum_{n\geq 1} \frac{1}{N(\cP)^{nz}}<\infty$, with $z> \frac{3}{2}$, we have the (in)equalities:
\begin{eqnarray*}
\sum_\cP\sum_{n\geq 1} \frac{1}{N(\cP)^{n(z-\frac{1}{2})}} & \leq & \sum_p \sum_{\cP\mid p}\sum_{n\geq 1} \frac{1}{p^{n(z-\frac{1}{2})}} \label{eq:star-1}\\
& \leq & [k:\bbQ] \cdot \sum_p\sum_{n\geq 1} \frac{1}{p^{n(z-\frac{1}{2})}}  \\
& = & [k:\bbQ] \cdot \sum_p \frac{1}{p^{(z-\frac{1}{2})}-1} \\
& \leq & 2 \cdot [k:\bbQ] \cdot \sum_p \frac{1}{p^{(z-\frac{1}{2})}} \nonumber \\
& \leq & 2 \cdot [k:\bbQ] \cdot \sum_{n\geq 1} \frac{1}{n^{(z-\frac{1}{2})}} < \infty \label{eq:star-3}\,.
\end{eqnarray*}
Let us now assume that $\mathrm{char}(k)>0$. Recall that in this case $k$ is a finite field extension of $\bbF_q(t)$. In what concerns $\sum_\cP\sum_{n\geq 1} \frac{1}{N(\cP)^{nz}}<\infty$, with $z>1$, we have the following (in)equalities ($\langle p(t)\rangle$, resp. $\langle q(t)\rangle$, stands for a prime ideal, resp. ideal, of the ring $\bbF_q[t]$): 
\begin{eqnarray}
\sum_\cP\sum_{n\geq 1} \frac{1}{N(\cP)^{nz}} & \leq & \sum_{\langle p(t) \rangle} \sum_{\cP\mid \langle p(t) \rangle}\sum_{n\geq 1} \frac{1}{N(\langle p(t)\rangle)^{nz}} \label{eq:star-111}\\
& \leq & [k:\bbF_q(t)] \cdot \sum_{\langle p(t) \rangle}\sum_{n\geq 1} \frac{1}{N(\langle p(t) \rangle)^{nz}} \label{eq:star-222}\\
& = & [k:\bbF_q(t)] \cdot \sum_{\langle p(t)\rangle} \frac{1}{N(\langle p(t) \rangle)^z-1} \label{eq:star-222-1} \\
& \leq & 2 \cdot [k;\bbF_q(t)] \cdot \sum_{\langle p(t) \rangle} \frac{1}{N(\langle p(t) \rangle)^z} \label{eq:star-333} \\
& \leq & 2 \cdot [k;\bbF_q(t)] \cdot \sum_{\langle q(t) \rangle} \frac{1}{N(\langle q(t) \rangle)^z} < \infty \label{eq:star-444}\,,
\end{eqnarray}
where \eqref{eq:star-111} follows from the fact that the norm $N(\cP)$ is a power of $N(\langle p(t) \rangle)$ whenever $\cP$ divides $\langle p(t)\rangle$, \eqref{eq:star-222} from the fact that the number of prime ideals $\cP$ which divide $\langle p(t)\rangle$ is always bounded by the degree $[k:\bbF_q(t)]$, \eqref{eq:star-222-1} from the (convergent) geometric series $\sum_{n\geq 0} \frac{1}{(N(\langle p(t) \rangle)^z)^n} =\frac{1}{1- \frac{1}{N(\langle p(t) \rangle)^z}}$, \eqref{eq:star-333} from a simple inspection, and \eqref{eq:star-444} from the fact that the classical zeta function $\sum_{\langle q(t) \rangle} \frac{1}{N(\langle q(t) \rangle)^s}$ is convergent when $\mathrm{Re}(s)>1$. Similarly, in what concerns $\sum_\cP\sum_{n\geq 1} \frac{1}{N(\cP)^{nz}}<\infty$, with $z > \frac{3}{2}$, we have the (in)equalities:
\begin{eqnarray*}
\sum_\cP\sum_{n\geq 1} \frac{1}{N(\cP)^{n(z-\frac{1}{2})}} & \leq & \sum_{\langle p(t) \rangle} \sum_{\cP\mid \langle p(t) \rangle}\sum_{n\geq 1} \frac{1}{N(\langle p(t)\rangle)^{n(z-\frac{1}{2})}} \label{eq:star-11}\\
& \leq & [k:\bbF_q(t)] \cdot \sum_{\langle p(t) \rangle}\sum_{n\geq 1} \frac{1}{N(\langle p(t) \rangle)^{n(z-\frac{1}{2})}} \\
& = & [k:\bbF_q(t)] \cdot \sum_{\langle p(t)\rangle} \frac{1}{N(\langle p(t) \rangle)^{(z-\frac{1}{2})}-1} \nonumber \\
& \leq & 2\cdot [k:\bbF_q(t)] \cdot \sum_{\langle p(t) \rangle } \frac{1}{N(\langle p(t) \rangle)^{(z-\frac{1}{2})}} \nonumber \\
& \leq & 2\cdot [k:\bbF_q(t)] \cdot \sum_{\langle q(t) \rangle} \frac{1}{N(\langle q(t) \rangle)^{(z-\frac{1}{2})}} < \infty \label{eq:star-33}\,.
\end{eqnarray*}
This concludes the proof of Lemma \ref{lem:key}.
\end{proof}
\section{Proof of Theorem \ref{thm:main}}
Note first that if the $L$-functions $\{L_w(X;s)\}_{w\,\mathrm{even}}$, resp. $\{L_w(X;s)\}_{w\,\mathrm{odd}}$, satisfy condition $(\mathrm{C}1)$, then the shifted $L$-functions $\{L_w(X;s+ \frac{w}{2})\}_{w\,\mathrm{even}}$, resp. $\{L_w(X;s+ \frac{w-1}{2})\}_{w\,\mathrm{odd}}$, also satisfy condition $(\mathrm{C}1)$. Thanks to Proposition \ref{prop:equalities}, this hence implies that the noncommutative $L$-function $L_{\mathrm{even}}(\perf_\dg(X);s)$, resp. $L_{\mathrm{odd}}(\perf_\dg(X);s)$, admits a (unique) meromorphic continuation to the entire complex plane. Now, Proposition \ref{prop:equalities} implies moreover the implications \eqref{eq:implications}.
 
Let us assume now that $\mathrm{char}(k)>0$ and that the conjecture $\mathrm{R}_{\mathrm{even}}(\perf_\dg(X))$, resp. $\mathrm{R}_{\mathrm{odd}}(\perf_\dg(X))$, holds. It follows from the work of Grothendieck \cite{Grothendieck} and Deligne \cite{Deligne, Weil1} that the $L$-function $L_w(X;s)$, with $0 \leq w \leq 2\mathrm{dim}(X)$, does not have a pole in the critical strip $\frac{w}{2} < \mathrm{Re}(s) < \frac{w}{2}+1$. Consequently, thanks to Proposition \ref{prop:equalities}, we conclude that all the zeros of the $L$-function $L_w(X;s)$ that are contained in the critical strip $\frac{w}{2} < \mathrm{Re}(s) < \frac{w}{2}+1$ lie necessarily in the vertical line $\mathrm{Re}(s)=\frac{w+1}{2}$ (otherwise, the noncommutative $L$-function $L_{\mathrm{even}}(\perf_\dg(X);s)$, resp. $L_{\mathrm{odd}}(\perf_\dg(X);s)$, would have a zero outside the vertical line $\mathrm{Re}(s)=\frac{1}{2}$, resp. $\mathrm{Re}(s)=1$). In other words, the converse implications of \eqref{eq:implications} hold.

Let us assume now that $\mathrm{char}(k)=0$, that the $L$-functions $\{L_w(X;s)\}_{0 \leq w\leq 2\mathrm{dim}(X)}$ satisfy condition $\mathrm{(C2)}$, and that the conjecture $\mathrm{R}_{\mathrm{even}}(\perf_\dg(X))$, resp. $\mathrm{R}_{\mathrm{odd}}(\perf_\dg(X))$, holds. Since the $L$-function $L_w(X;s)$, with $0 \leq w \leq 2\mathrm{dim}(X)$, does not have a pole in the critical strip $\frac{w}{2} < \mathrm{Re}(s) < \frac{w}{2}+1$, we hence conclude from Proposition \ref{prop:equalities} that all the zeros of $L_w(X;s)$ that are contained in the critical strip $\frac{w}{2} < \mathrm{Re}(s) < \frac{w}{2}+1$ lie necessarily in the vertical line $\mathrm{Re}(s)=\frac{w+1}{2}$ (otherwise, the noncommutative $L$-function $L_{\mathrm{even}}(\perf_\dg(X);s)$, resp. $L_{\mathrm{odd}}(\perf_\dg(X);s)$, would have a zero outside the vertical line $\mathrm{Re}(s)=\frac{1}{2}$, resp. $\mathrm{Re}(s)=1$). In other words, the converse implications of \eqref{eq:implications} also hold.

\section{Proof of Theorem \ref{thm:HPD}}\label{}
By definition of the Lefschetz decomposition $\perf(X)=\langle \bbA_0, \bbA_1(1), \ldots, \bbA_{i-1}(i-1)\rangle$, we have a chain of admissible triangulated subcategories $\bbA_{i-1}\subseteq \cdots \subseteq \bbA_1\subseteq \bbA_0$ with $\bbA_r(r):=\bbA_r\otimes \cL_X(r)$; note that $\bbA_r(r)$ is equivalent to $\bbA_r$. Let us write $\mathfrak{a}_r$ for the right-orthogonal complement of $\bbA_{r+1}$ in $\bbA_r$; these are called the {\em primitive subcategories} in \cite[\S4]{KuznetsovHPD}. By construction, we have the following semi-orthogonal decompositions:
\begin{eqnarray}\label{eq:semi1}
\bbA_r = \langle \mathfrak{a}_r, \mathfrak{a}_{r+1}, \ldots, \mathfrak{a}_{i-1}\rangle && 0 \leq r \leq i-1\,.
\end{eqnarray}
Following \cite[Thm.~6.3]{KuznetsovHPD}, we have a chain of admissible triangulated subcategories $\bbB_{j-1}\subseteq \cdots \subseteq \bbB_1 \subseteq \bbB_0$ and an associated HP-dual Lefschetz decomposition $\perf(Y)=\langle \bbB_{j-1}(1-j), \ldots, \bbB_1(-1), \bbB_0\rangle$ with respect to $\cL_Y(1)$. Moreover, we have the following semi-orthgonal decompositions:
\begin{eqnarray}\label{eq:semi2}
\bbB_r = \langle \mathfrak{a}_0, \mathfrak{a}_{1}, \ldots, \mathfrak{a}_{\mathrm{dim}(V) - r- 2}\rangle && 0 \leq r \leq j-1\,.
\end{eqnarray}
Furthermore, since the linear subspace $L\subset V^\vee$ is generic, we can assume without loss of generality that the linear sections $X_L$ and $Y_L$ are not only smooth but also that they have the {\em expected dimensions}, i.e., $\mathrm{dim}(X_L)=\mathrm{dim}(X) - \mathrm{dim}(L)$ and $\mathrm{dim}(Y_L)=\mathrm{dim}(Y) - \mathrm{dim}(L^\perp)$. As explained in \cite[Thm.~6.3]{KuznetsovHPD}, this yields the following semi-orthogonal decompositions
\begin{eqnarray}
\perf(X_L) & = & \langle \cC_L, \bbA_{\mathrm{dim}(L)}(1), \ldots, \bbA_{i-1}(i- \mathrm{dim}(L)) \rangle \label{eq:semi3} \\
\perf(Y_L) & =&  \langle \bbB_{j-1}(\mathrm{dim}(L^\perp)-j), \ldots, \bbB_{\mathrm{dim}(L^\perp)}(-1), \cC_L\rangle\,, \label{eq:semi4}
\end{eqnarray}
where $\cC_L$ is a common (triangulated) category. Let us denote by $\bbA_{r, \dg}$, by $\mathfrak{a}_{r, \dg}$, and by $\cC_{L, \dg}$, the dg enhancements of $\bbA_r$, $\mathfrak{a}_r$, and $\cC_L$, induced from the dg category $\perf_\dg(X_L)$. Similarly, let us denote by $\bbB_{r, \dg}$ and by $\cC'_{L, \dg}$ the dg enhancements of $\bbB_r$ and $\cC_L$ induced from the dg category $\perf_\dg(Y_L)$. Since the functor $\perf(X_L) \to \cC_L \to \perf(Y_L)$ is of Fourier-Mukai type, the dg categories $\cC_{L, \dg}$ and $\cC'_{L, \dg}$ are Morita equivalent. Note that, by construction, all the aforementioned dg categories are geometric noncommutative $k$-schemes. Note that by combining the above semi-orthogonal decompositions \eqref{eq:semi1}-\eqref{eq:semi4} with Lemma \ref{lem:computation} below and with an iterated application of Proposition \ref{prop:product}, we obtain the following equalities  
\begin{eqnarray}
L_{\mathrm{even}}(\perf_\dg(X_L);s)=L_{\mathrm{even}}(\cC_{L,\dg};s) \cdot \zeta_k(s) \cdots \zeta_k(s) && L_{\mathrm{odd}}(\perf_\dg(X_L);s)=L_{\mathrm{odd}}(\cC_{L,\dg};s) \label{eq:equality-verylast1} \\
L_{\mathrm{even}}(\perf_\dg(Y_L);s)= \zeta_k(s) \cdots \zeta_k(s) \cdot L_{\mathrm{even}}(\cC_{L,\dg};s)  && L_{\mathrm{odd}}(\perf_\dg(Y_L);s)=L_{\mathrm{odd}}(\cC_{L,\dg};s)\,, \label{eq:equality-verylast2}
\end{eqnarray}
where the number of copies of the Dedekind zeta function $\zeta_k(s)$ in \eqref{eq:equality-verylast1}, resp. in \eqref{eq:equality-verylast2}, is equal to the sum of the ranks of the (free) Grothendieck groups $K_0(\bbA_{\mathrm{dim}(L)}), \ldots, K_0(\bbA_{i-1})$, resp. $K_0(\bbB_{j-1}), \ldots, K_0(\bbB_{\mathrm{dim}(L^\perp)})$. Consequently, since the Dedekind zeta function $\zeta_k(s)$ satisfies conditions $\mathrm{(C1)}\text{-}\mathrm{(C2)}$, the following holds:
\begin{eqnarray}\label{eq:quasi-final}
\mathrm{ERH}_k \Rightarrow \big( \mathrm{R}_{\mathrm{even}}(\mathrm{perf}_\dg(X_L)) \Leftrightarrow  \mathrm{R}_{\mathrm{odd}}(\mathrm{perf}_\dg(Y_L)) \big) && \mathrm{R}_{\mathrm{odd}}(\mathrm{perf}_\dg(X_L))\Leftrightarrow  \mathrm{R}_{\mathrm{odd}}(\mathrm{perf}_\dg(Y_L))\,.
\end{eqnarray}
The proof follows now from the combination of \eqref{eq:quasi-final} with Theorem \ref{thm:main}.
\begin{lemma}\label{lem:computation}
We have the following computations (with $0\leq r \leq i-1$)
\begin{eqnarray}\label{eq:computation-final}
L_{\mathrm{even}}(\mathfrak{a}_{r,\dg};s)= \underbrace{\zeta_k(s) \cdots \zeta_k(s)}_{n_r} && L_{\mathrm{odd}}(\mathfrak{a}_{r, \dg}; s)=1\,,
\end{eqnarray}
where $n_r$ stands for the rank of the (free) Grothendieck group $K_0(\mathfrak{a}_r)$.
\end{lemma}
\begin{proof}
We start by recalling from \cite[\S2.3]{book} the definition of the additive category $\Hmo_0(k)$; consult \S\ref{sub:additive}. Given two dg categories $\cA$ and $\cB$, let us write $\cD(\cA^\op\otimes_k \cB)$ for the derived category of dg $\cA\text{-}\cB$-bimodules and $\mathrm{rep}(\cA,\cB)$ for the full triangulated subcategory of $\cD(\cA^\op \otimes_k \cB)$ consisting of those dg $\cA\text{-}\cB$-bimodules $\mathrm{B}$ such that for every object $x \in \cA$ the associated right dg $\cB$-module $\mathrm{B}(x,-)$ belongs to the full triangulated subcategory of compact objects $\cD_c(\cB)$. The objects of the category $\Hmo_0(k)$ are the (small) dg categories, the abelian groups of morphisms $\mathrm{Hom}_{\Hmo_0(k)}(U(\cA), U(\cB))$ are given by the Grothendieck groups $K_0\mathrm{rep}(\cA,\cB)$, and the composition law is induced by the (derived) tensor product of dg bimodules. 

By assumption, we have a full exceptional collection $\bbA_0 =\langle \cE_1, \ldots, \cE_n\rangle$. As explained in \cite[\S2.4.2]{book}, this implies that $U(\bbA_{0, \dg})\simeq U(k)^{\oplus n}$ in the additive category $\mathrm{Hmo}_0(k)$. Moreover, since $\mathfrak{a}_r$ is an admissible triangulated subcategory of $\bbA_0$, $U(\mathfrak{a}_{r, \dg})$ becomes a direct summand of $U(\bbA_{0,\dg})$. Thanks to the computation
\begin{eqnarray*}
\Hom_{\Hmo_0(k)}(U(\bbA_0),U(\bbA_0)) &\simeq &\Hom_{\Hmo_0(k)}(U(k)^{\oplus n}, U(k)^{\oplus n}) \\
&  \simeq & \Hom_{\Hmo_0(k)}(U(k),U(k))^{\oplus (n \times n)} \\
& = & K_0\rep(k,k)^{\oplus (n \times n)} \simeq K_0\cD_c(k)^{\oplus (n\times n)} \simeq \bbZ^{\oplus (n \times n)}\,,
\end{eqnarray*}
we observe that the endomorphism ring of $U(\bbA_{0, \dg})$ is isomorphic to the ring of $(n\times n)$-matrices with $\bbZ$-coefficients. This implies that $U(\mathfrak{a}_{r, \dg})$ is then necessarily isomorphic to $U(k)^{\oplus n_r}$ for a certain integer $n_r \leq n$. Moreover, this integer $n_r$ is equal to the rank of the (free) Grothendieck group $K_0(\mathfrak{a}_r)$ because 
\begin{eqnarray*}
\Hom_{\Hmo_0(k)}(U(k), U(\mathfrak{a}_{r, \dg})):=K_0\rep(k, \mathfrak{a}_{r, \dg}) \simeq K_0(\mathfrak{a}_r)&\mathrm{and} & \Hom_{\Hmo_0(k)}(U(k), U(k)^{\oplus n_r}) \simeq \bbZ^{\oplus n_r}\,.
\end{eqnarray*}
Now, recall from Proposition \ref{prop:additive} and Lemma \ref{lem:extension1} that, given any non-archimedean place $\nu \in \Sigma_k$, the associated functors $\mathrm{H}_{\mathrm{even}, l_\nu}(-\otimes_k k_\nu)$ and $\mathrm{H}_{\mathrm{odd}, l_\nu}(-\otimes_k k_\nu)$ are additive invariants. Consequently, since these additive invariants factor through the additive category $\mathrm{Hmo}_0(k)$ (consult the factorization \eqref{eq:factorization}), we obtain an isomorphism of $\bbQ_{l_\nu}$-linear $\mathrm{Gal}(\overline{k_\nu}/k_\nu)$-modules between $\mathrm{H}_{\mathrm{even}, l_\nu}(\mathfrak{a}_{r, \dg}\otimes_k k_\nu)$ and the direct sum $\mathrm{H}_{\mathrm{even}, l_\nu}(k\otimes_k k_\nu)^{\oplus n_r}$ and between $\mathrm{H}_{\mathrm{odd}, l_\nu}(\mathfrak{a}_{r, \dg}\otimes_k k_\nu)$ and the direct sum $\mathrm{H}_{\mathrm{odd}, l_\nu}(k\otimes_k k_\nu)^{\oplus n_r}$. Therefore, making use of the canonical Morita equivalence $k \to \perf_\dg(\mathrm{Spec}(k))$, of Proposition \ref{prop:etale}, and of Definition \ref{def:L-functions}, we obtain the above computations \eqref{eq:computation-final}.
\end{proof}

\section{Proof of Theorems \ref{thm:gluing}, \ref{thm:CY}, \ref{thm:finite1}, and \ref{thm:finite2}}\label{sec:proof}
\noindent
{\bf Proof of Theorem \ref{thm:gluing}.} By construction of the noncommutative gluing $X\odot Y$, we have the semi-orthogonal decomposition $H^0(X\odot_{\mathrm{B}}Y)=\langle H^0(\perf_\dg(X)), H^0(\perf_\dg(Y))\rangle$. Consequently, thanks to Proposition \ref{prop:product}, we conclude that $L_{\mathrm{even}}(X\odot_{\mathrm{B}}Y;s)=L_{\mathrm{even}}(\perf_\dg(X);s)\cdot L_{\mathrm{even}}(\perf_\dg(Y);s)$. This yields the implication $\mathrm{R}_{\mathrm{even}}(\perf_\dg(X)) + \mathrm{R}_{\mathrm{even}}(\perf_\dg(X)) \Rightarrow \mathrm{R}_{\mathrm{even}}(X\odot_{\mathrm{B}}Y)$. Therefore, the proof follows now from Theorem \ref{thm:main}. The proof of the odd case is similar: simply replace the word ``even'' by the word ``odd''.

\medskip

\noindent
{\bf Proof of Theorem \ref{thm:CY}.} Thanks to the decomposition $\perf(X)=\langle \cT, \cO_X, \ldots, \cO_X(n-\mathrm{deg}(X)) \rangle$, an iterated application of Proposition \ref{prop:product} yields the following equalities:
\begin{eqnarray}\label{eq:equalities-last}
L_{\mathrm{even}}(\perf_\dg(X);s)=L_{\mathrm{even}}(\cT_\dg; s)\cdot \underbrace{\zeta_k(s) \cdots \zeta_k(s)}_{n-\mathrm{deg}(X)+1} && L_{\mathrm{odd}}(\perf_\dg(X);s)= L_{\mathrm{odd}}(\cT_\dg;s)\,.
\end{eqnarray}
On the one hand, since the Dedekind zeta function $\zeta_k(s)$ satisfies conditions $\mathrm{(C1)}\text{-}\mathrm{(C2)}$, the left-hand side of \eqref{eq:equalities-last} implies that $\mathrm{R}_{\mathrm{even}}(\perf_\dg(X)) \Rightarrow \mathrm{R}_{\mathrm{even}}(\cT_\dg)$. On the other hand, the right-hand side of \eqref{eq:equalities-last} implies that $\mathrm{R}_{\mathrm{odd}}(\perf_\dg(X)) \Leftrightarrow \mathrm{R}_{\mathrm{odd}}(\cT_\dg)$. Consequently, the proof follows now from Theorem \ref{thm:main}.

\medskip

\noindent
{\bf Proof of Theorem \ref{thm:finite1}.} As proved in \cite[Thm.~3.5]{Azumaya}, since the quotient $k$-algebra $A/J$ is separable, we have an isomorphism $U(A)_\bbQ \simeq U(k_1)_\bbQ \oplus \cdots \oplus U(k_n)_\bbQ$ in the $\bbQ$-linearization $\Hmo_0(k)_\bbQ$ of the additive category $\Hmo_0(k)$; consult \S\ref{sub:additive}. Thanks to Proposition \ref{prop:additive} and Lemma \ref{lem:extension1}, given any non-archimedean place $\nu \in \Sigma_k$, the associated functors $\mathrm{H}_{\mathrm{even}, l_\nu}(-\otimes_k k_\nu)$ and $\mathrm{H}_{\mathrm{odd}, l_\nu}(-\otimes_k k_\nu)$ are additive invariants. Since these additive invariants take values in the category $\mathrm{Gal}(\overline{k_\nu}/k_\nu)\text{-}\mathrm{Mod}$, which is a $\bbQ$-linear category, the above isomorphism combined with the factorization \eqref{eq:factorization} leads to the following isomorphisms of $\mathrm{Gal}(\overline{k_\nu}/k_\nu)$-modules:
\begin{eqnarray*}
\mathrm{H}_{\mathrm{even}, l_\nu}(A\otimes_k k_\nu)\simeq \bigoplus_{i=1}^n\mathrm{H}_{\mathrm{even}, l_\nu}(k_i\otimes_k k_\nu) && \mathrm{H}_{\mathrm{odd}, l_\nu}(A\otimes_k k_\nu)\simeq \bigoplus_{i=1}^n\mathrm{H}_{\mathrm{odd}, l_\nu}(k_i\otimes_k k_\nu)\,.
\end{eqnarray*} 
Making use of the canonical Morita equivalences $k_i \to \perf_\dg(\mathrm{Spec}(k_i))$, of Proposition \ref{prop:etale}, and of Definition \ref{def:L-functions}, we hence conclude that $L_{\mathrm{even}}(A;s)=\prod_{i=1}^n L_0(\mathrm{Spec}(k_i))$ and $L_{\mathrm{odd}}(A;s)=1$. This yields the implication $\sum_{i=1}^n \mathrm{R}_0(\mathrm{Spec}(k_i)) \Rightarrow \mathrm{R}_{\mathrm{even}}(A)$.

\medskip

\noindent
{\bf Proof of Theorem \ref{thm:finite2}.} Recall first that we have a canonical Morita equivalence $A\to \perf_\dg(A)$. Following Orlov \cite[Cor.~3.4]{Orlov1}, there exists a smooth proper $k$-scheme $X$ such that $\perf(A)$ is an admissible triangulated subcategory of $\perf(X)$. Moreover, since the quotient dg $k$-algebra $A/J_+$ is separable, we have a semi-orthogonal decomposition $\perf(X)=\langle \perf(D_1), \ldots, \perf(D_n)\rangle$, where $D_1, \ldots, D_n$ are separable division $k$-algebras. Let us denote by $k_1, \ldots, k_n$ the centers of the $k$-algebras $D_1, \ldots, D_n$; these are separable finite field extensions of $k$. Making use of \cite[Thm.~2.11]{Azumaya}, we hence conclude that $U(A)_\bbQ$ becomes a direct summand of $U(\perf_\dg(X))_\bbQ \simeq U(k_1)_\bbQ \oplus \cdots \oplus U(k_n)_\bbQ$ in the category $\Hmo_0(k)_\bbQ$. Note that, similarly to the proof of Theorem \ref{thm:finite1}, this implies that $L_{\mathrm{odd}}(A;s)=1$. Now, recall from \cite[\S4.9]{book} that the classical category of Artin motives (with $\bbQ$-coefficients) may be identified with the idempotent completion of the full subcategory of $\Hmo_0(k)_\bbQ$ consisting of the objects $\{U(\perf_\dg(X))_\bbQ\,|\,X \,\,\mathrm{is}\,\,\mathrm{a}\,\,0\text{-}\mathrm{dimensional}\,\,k\text{-}\mathrm{scheme}\}$. Under this identification, $U(A)_\bbQ$ corresponds to an Artin motive $\rho\colon \mathrm{Gal}(\overline{k}/k)\to \mathrm{GL}(V)$, where $V$ is a finite dimensional $\bbQ$-vector space, and the noncommutative $L$-function $L_{\mathrm{even}}(A;s)$ to the classical Artin $L$-function $L(\rho;s)$ of $\rho$. Consequently, the proof follows now from Weil's work \cite[\S V]{Weil-book}.

\end{document}

\end{proof}